\documentclass[12pt]{amsart}

\usepackage[letterpaper,margin=1in]{geometry}
\usepackage{amsmath,amssymb,amsthm,mathtools}
\usepackage[T1]{fontenc}
\usepackage{lmodern}
\usepackage{microtype}
\usepackage{hyperref}
\usepackage{enumitem}

\hypersetup{colorlinks=true,linkcolor=black,citecolor=black,urlcolor=black}
\allowdisplaybreaks

\newtheorem{theorem}{Theorem}[section]
\newtheorem{lemma}[theorem]{Lemma}
\newtheorem{proposition}[theorem]{Proposition}
\newtheorem{corollary}[theorem]{Corollary}
\theoremstyle{remark}
\newtheorem{remark}[theorem]{Remark}

\newcommand{\Z}{\mathbb Z}
\newcommand{\Q}{\mathbb Q}

\newcommand{\C}{\mathbb C}
\newcommand{\F}{\mathbb F}
\newcommand{\Hh}{\mathbb H}
\newcommand{\Res}{\operatorname{Res}}
\newcommand{\CT}{\operatorname{CT}}
\newcommand{\Sym}{\operatorname{Sym}}
\newcommand{\Ha}{\operatorname{Ha}}

\title{Supercongruences for Domb Numbers}
\author{Kevin Calderon}
\date{2026}
\address{Departamento de Matem\'aticas, Facultad de Ciencias, Universidad Nacional Aut\'onoma de M\'exico, 04510 Ciudad de M\'exico, Mexico}
\email{kevincalderon@ciencias.unam.mx}

\begin{document}
\maketitle

\begin{abstract}
We prove two supercongruence statements proposed by Z.-W. Sun for Domb numbers. The first concerns the unweighted truncated sum
\[
S_p=\sum_{n=0}^{p-1}D_n
\]
and is governed by the two proper ideal classes of the order of discriminant $-60$ in $\Q(\sqrt{-15})$. We prove Sun's conjectured congruences modulo $p^2$ and obtain stronger formulas modulo $p^3$ in the split cases. The second concerns Domb sums weighted by the Lucas sequences associated with $X^2-11X+1$. Sun's original formulation has the exceptional counterexample $p=3$; after excluding this prime, we prove the corrected statement, including its modulo-$p^3$ refinement when both relevant quadratic characters are $+1$.

Both results arise from one modular mechanism. The twisted Domb series is a modular period on $\Gamma_0(6)+\langle 3\rangle$ and is related by a quadratic pullback to Zagier's sporadic $C$-sequence. At primes inert in the imaginary quadratic field underlying the relevant complex multiplication, the truncated $C$-period is the Hasse invariant and vanishes at supersingular reduction; a finite pullback then gives divisibility modulo $p^2$. At split primes, explicitly constructed determinant-$1$ or determinant-$3$ matrices identify the Beukers unit root, yielding congruences modulo $p^3$. The two applications correspond respectively to orders of complex multiplication of discriminants $-60$ and $-156$.
\end{abstract}

\section{Introduction}

For $n\ge 0$, the Domb numbers are
\begin{equation}
D_n=\sum_{k=0}^n \binom{n}{k}^2\binom{2k}{k}\binom{2n-2k}{n-k}.
\tag{1.1}
\end{equation}
They begin
\[
D_0=1,\qquad D_1=4,\qquad D_2=28,\qquad D_3=256,\qquad D_4=2716,\ldots.
\]
Historically, these integers arose in the combinatorics of crystal lattices in Domb's study of cooperative phenomena in statistical mechanics \cite{Domb}. In particular, they are naturally connected with closed lattice configurations on the diamond lattice. This physical origin suggests that the generating series should be regarded as a lattice Green-function-type period rather than merely as a formal power series. There is a second and different description of the same integers. Introduce the ordinary generating function
\begin{equation}
D(z)=\sum_{n=0}^{\infty}D_nz^n.
\tag{1.2}
\end{equation}
The Domb numbers satisfy
\begin{equation}
(n+1)^3D_{n+1}=2(2n+1)(5n^2+5n+2)D_n-64n^3D_{n-1},
\tag{1.3}
\end{equation}
with $D_0=1$ and $D_1=4$. If $\theta_z=z\,d/dz$, then (1.3) is equivalent to
\begin{equation}
\Bigl[\theta_z^3-2z(2\theta_z+1)(5\theta_z^2+5\theta_z+2)+64z^2(\theta_z+1)^3\Bigr]D(z)=0.
\tag{1.4}
\end{equation}
At first sight the order-three equation (1.4) might suggest an intrinsically rank-three origin. The essential fact is that the equation is a symmetric square. A classical hypergeometric representation is
\begin{equation}
D(z)=\frac{1}{1-4z}\,{}_3F_2\!\left(\begin{matrix}\frac13,\frac12,\frac23\\1,1\end{matrix};\frac{108z^2}{(1-4z)^3}\right).
\tag{1.5}
\end{equation}
Clausen's identity gives
\[
{}_3F_2\!\left(\begin{matrix}\frac13,\frac12,\frac23\\1,1\end{matrix};u\right)
=\left[{}_2F_1\!\left(\begin{matrix}\frac16,\frac13\\1\end{matrix};u\right)\right]^2.
\]
Hence
\begin{equation}
D(z)=H(z)^2,
\tag{1.6}
\end{equation}
where
\begin{equation}
H(z)=\frac{1}{\sqrt{1-4z}}\,{}_2F_1\!\left(\begin{matrix}\frac16,\frac13\\1\end{matrix};\frac{108z^2}{(1-4z)^3}\right).
\tag{1.7}
\end{equation}
The function $H$ satisfies
\begin{equation}
\Bigl[\theta_z^2-2z(10\theta_z^2+5\theta_z+1)+16z^2(2\theta_z+1)^2\Bigr]H(z)=0.
\tag{1.8}
\end{equation}
Thus the third-order Domb equation is the symmetric square of a rank-two period equation, after the indicated pullback and normalization. For a one-parameter family of elliptic curves, the first cohomology has rank two and its periods satisfy a second-order Picard--Fuchs equation. In the present situation, $H(z)$ is a normalized elliptic period, and (1.6) places $D(z)$ in its symmetric square. The passage from order two to order three is the representation-theoretic identity
\[
\dim\Sym^2(\C^2)=3.
\]

The original lattice interpretation and the period interpretation are compatible. If $P$ is a Laurent polynomial and $a_n=\CT(P^n)$, then
\[
\sum_{n\ge0}a_nt^n=\CT\!\left(\frac{1}{1-tP}\right),
\]
and the constant term can be written as a torus integral,
\[
\sum_{n\ge0}a_nt^n=\frac{1}{(2\pi i)^r}
\int_{|z_1|=\cdots=|z_r|=1}
\frac{1}{1-tP(z_1,\ldots,z_r)}\frac{dz_1}{z_1}\cdots\frac{dz_r}{z_r}.
\]
Such an integral is a period of an algebraic family.

Let $q=e^{2\pi i\tau}$, $\tau\in\Hh$, and let $\eta$ denote Dedekind's eta function. Define
\begin{equation}
t(\tau)=\left(\frac{\eta(2\tau)\eta(6\tau)}{\eta(\tau)\eta(3\tau)}\right)^6
\tag{1.9}
\end{equation}
and
\begin{equation}
h(\tau)=\frac{\eta(\tau)^2\eta(3\tau)^2}{\eta(2\tau)\eta(6\tau)}.
\tag{1.10}
\end{equation}
The twisted Domb series
\begin{equation}
B(t)=\sum_{n\ge0}(-1)^nD_nt^n=D(-t)
\tag{1.11}
\end{equation}
satisfies the modular identity
\begin{equation}
B(t(\tau))=h(\tau)^2.
\tag{1.12}
\end{equation}
Equivalently,
\[
\sum_{n=0}^{\infty}(-1)^nD_n
\left(\frac{\eta(2\tau)\eta(6\tau)}{\eta(\tau)\eta(3\tau)}\right)^{6n}
=
\left(\frac{\eta(\tau)^2\eta(3\tau)^2}{\eta(2\tau)\eta(6\tau)}\right)^2.
\]
Thus the Domb numbers are not the Fourier coefficients of the modular form on the right. Rather, they are the Taylor coefficients obtained when this weight-two modular period is expanded in the natural level-$6$ coordinate $t$. The square root $h(\tau)$ is a weight-one modular period, while $h(\tau)^2$ has weight two. Formula (1.12) is the modular counterpart of (1.6).

The form needed for the supercongruence problem is a second level-$6$ description. Let
\begin{equation}
c_n=\sum_{k=0}^n\binom{n}{k}^2\binom{2k}{k},\qquad
A(x)=\sum_{n\ge0}c_nx^n,
\tag{1.13}
\end{equation}
and put
\begin{equation}
\Delta(x)=(1-x)(1-9x),\qquad w(x)=\frac{x}{\Delta(x)}.
\tag{1.14}
\end{equation}
Then
\begin{equation}
B(w(x))=\Delta(x)A(x)^2.
\tag{1.15}
\end{equation}
This identity transfers Frobenius information for the rank-two elliptic family attached to $A$ to the twisted Domb period $B$.

The symmetric-square structure has a parallel interpretation in characteristic $p$. For an elliptic curve over a finite field, Frobenius acts on its first cohomology with two eigenvalues $\alpha_p,\beta_p$. On the symmetric square the eigenvalues become
\[
\alpha_p^2,\qquad \alpha_p\beta_p,\qquad \beta_p^2.
\]
Thus the same rank-two-to-rank-three passage that appears analytically in (1.6) also governs the $p$-adic arithmetic of the Domb period. A second ingredient is complex multiplication. If $\tau_0\in\Hh$ is imaginary quadratic, then $\C/(\Z+\tau_0\Z)$ has complex multiplication by an order in an imaginary quadratic field. At such parameters, modular functions take algebraic values and the Frobenius behavior is controlled by the splitting of $p$ in that field. At a prime inert in the imaginary quadratic field underlying the complex multiplication, good reduction is supersingular. The Hasse invariant therefore vanishes. In the level-$6$ family considered here, the truncated $C$-period $A_p$ is the Hasse scalar, and the finite form of (1.15) forces the associated truncated Domb period to vanish modulo $p^2$. At a split prime, the reduction is ordinary and one Frobenius eigenvalue is a $p$-adic unit. Beukers' excellent-Frobenius formalism identifies the truncated Domb period, at the relevant point of complex multiplication, with the square of this unit root modulo $p^3$. The determinant-$1$ and determinant-$3$ matrices constructed below encode the two possible ideal classes and the Atkin--Lehner action. For the unweighted sum the relevant order has discriminant $-60$ in $\Q(\sqrt{-15})$. For the Lucas-weighted sums the relevant order has discriminant $-156$ in $\Q(\sqrt{-39})$.

\subsection{Notation and the two statements}

For a prime $p$ define
\begin{equation}
S_p:=\sum_{n=0}^{p-1}D_n,\qquad
P_p(X):=\sum_{n=0}^{p-1}D_nX^n.
\tag{1.16}
\end{equation}
We also use the Lucas sequences attached to $X^2-11X+1$,
\begin{equation}
u_0=0,\quad u_1=1,\quad u_{n+1}=11u_n-u_{n-1},\qquad
v_0=2,\quad v_1=11,\quad v_{n+1}=11v_n-v_{n-1},
\tag{1.17}
\end{equation}
and the weighted sums
\begin{equation}
U_p:=\sum_{n=0}^{p-1}D_nu_n,\qquad
V_p:=\sum_{n=0}^{p-1}D_nv_n.
\tag{1.18}
\end{equation}
Put
\begin{equation}
\alpha=\frac{11+3\sqrt{13}}2,\qquad
\beta=\frac{11-3\sqrt{13}}2.
\tag{1.19}
\end{equation}
Then
\[
\alpha+\beta=11,\qquad \alpha\beta=1,\qquad \alpha-\beta=3\sqrt{13}.
\]
For later use, define
\begin{equation}
A_p(x)=\sum_{n=0}^{p-1}c_nx^n,\qquad
B_p(t)=\sum_{n=0}^{p-1}(-1)^nD_nt^n.
\tag{1.20}
\end{equation}

Sun's Conjecture 76 \cite{SunOpen} is Corollary \ref{cor:unweighted}. The following theorem adds the modulo-$p^3$ split cases.

\begin{theorem}\label{thm:unweighted}
Let $p>5$ be prime.
\begin{enumerate}[label=(\roman*)]
\item If $p=x^2+15y^2$, then
\begin{equation}
S_p\equiv 4x^2-2p-\frac{p^2}{4x^2}\pmod{p^3}.
\tag{1.21}
\end{equation}
\item If $p=3x^2+5y^2$, then
\begin{equation}
S_p\equiv 2p-12x^2+\frac{p^2}{12x^2}\pmod{p^3}.
\tag{1.22}
\end{equation}
\item If $\left(\frac{-15}{p}\right)=-1$, then
\begin{equation}
S_p\equiv0\pmod{p^2}.
\tag{1.23}
\end{equation}
\end{enumerate}
The denominators in (1.21) and (1.22) are interpreted modulo $p^3$; they are invertible because $p>5$ and $p\nmid x$.
\end{theorem}

\begin{corollary}\label{cor:unweighted}
For every prime $p>5$,
\[
S_p\equiv
\begin{cases}
4x^2-2p\pmod{p^2},&p=x^2+15y^2,\\
2p-12x^2\pmod{p^2},&p=3x^2+5y^2,\\
0\pmod{p^2},&\left(\frac{-15}{p}\right)=-1.
\end{cases}
\]
\end{corollary}

Sun's Lucas-weighted conjecture \cite{SunMO} is the following theorem for $p>3$; Remark \ref{rem:p3} explains the exceptional prime $p=3$.

\begin{theorem}\label{thm:lucas}
Let $p>3$ be prime.
\begin{enumerate}[label=(\roman*)]
\item If at least one of $\left(\frac p3\right)$ and $\left(\frac p{13}\right)$ is equal to $1$, then
\begin{equation}
U_p\equiv0\pmod{p^2}.
\tag{1.24}
\end{equation}
If both symbols are equal to $1$, then in fact
\begin{equation}
U_p\equiv0\pmod{p^3}.
\tag{1.25}
\end{equation}
\item If $\left(\frac p{39}\right)=-1$, then
\begin{equation}
V_p\equiv0\pmod{p^2}.
\tag{1.26}
\end{equation}
\end{enumerate}
\end{theorem}

\begin{remark}\label{rem:p3}
The prime $p=3$ is a counterexample to Sun's original formulation of part (i). Indeed,
\[
D_0=1,\qquad D_1=4,\qquad D_2=28,
\]
and
\[
u_0=0,\qquad u_1=1,\qquad u_2=11.
\]
Since $\left(\frac3{13}\right)=1$, the original statement would require $U_3\equiv0\pmod9$, whereas
\[
U_3=4+28\cdot11=312\equiv6\pmod9.
\]
Thus $p=3$ must be excluded.
\end{remark}

The two quadratic-form conditions in Theorems \ref{thm:unweighted} and \ref{thm:lucas} will be explained by two specializations at points of complex multiplication in the same twisted Domb modular family.

\section{Common modular and Frobenius machinery}

The period
\[
A(x)=\sum_{n\ge0}c_nx^n
\]
is attached to $\Gamma_0(6)$, whereas the twisted Domb period
\[
B(t)=\sum_{n\ge0}(-1)^nD_nt^n
\]
is attached to the Atkin--Lehner extension $\Gamma_0(6)+\langle3\rangle$. Beukers' level-$6$ parametrizations are
\begin{equation}
x(\tau)=\frac{h_6(\tau)}{1+9h_6(\tau)},\qquad
t(\tau)=\left(\frac{\eta(2\tau)\eta(6\tau)}{\eta(\tau)\eta(3\tau)}\right)^6,
\tag{2.1}
\end{equation}
with
\begin{equation}
t(\tau)=w(x(\tau)).
\tag{2.2}
\end{equation}

\begin{theorem}\label{thm:frob}
For primes $p\nmid6$, the normalized periods satisfy
\begin{equation}
A(x)\equiv A_p(x)A(x^\sigma)\pmod{p^2},
\tag{2.3}
\end{equation}
and
\begin{equation}
B(t)\equiv B_p(t)B(t^\sigma)\pmod{p^3}.
\tag{2.4}
\end{equation}
The Frobenius lifts are induced by $q\mapsto q^p$; in particular, if $t=w(x)$, then $t^\sigma=w(x^\sigma)$.
\end{theorem}

\noindent\emph{Proof.}
This is the specialization of Beukers' excellent-Frobenius theorems \cite[Theorems 1.4 and 1.13]{Beukers} to Zagier's sporadic $C$-sequence and to the twisted Domb family, together with the common level-$6$ parametrization recorded in the appendices of \cite{Beukers}. \hfill$\square$

\begin{theorem}\label{thm:straub}
For the sequence $c_n$ there exist derivative coefficients $c_n'$ such that, for $p\ge3$, $n\ge0$, and $0\le k<p$,
\begin{equation}
c_{pn+k}\equiv c_kc_n+pn\,c_k'c_n\pmod{p^2}.
\tag{2.5}
\end{equation}
\end{theorem}

\noindent\emph{Proof.}
Equation (2.5) follows from Straub \cite[Theorem 1.3]{Straub} for Zagier's $C$-sequence; the proof there gives the range $0\le k<p$. \hfill$\square$

\begin{lemma}\label{lem:pullback}
With $\Delta$ and $w$ as in (1.14),
\begin{equation}
B(w(x))=\Delta(x)A(x)^2.
\tag{2.6}
\end{equation}
\end{lemma}

\noindent\emph{Proof.}
The period $A(x)$ is the normalized holomorphic solution at $x=0$ of
\begin{equation}
\Bigl[\theta_x^2-x(10\theta_x^2+10\theta_x+3)+9x^2(\theta_x+1)^2\Bigr]A=0,
\qquad \theta_x=x\frac{d}{dx},
\tag{2.7}
\end{equation}
whereas $B(t)$ is the normalized holomorphic solution of
\begin{equation}
\Bigl[(1+20t+64t^2)\theta_t^3+(30t+192t^2)\theta_t^2+(18t+192t^2)\theta_t+4t+64t^2\Bigr]B=0.
\tag{2.8}
\end{equation}
These are the level-$6$ operators recorded by Beukers \cite{Beukers}. Substituting $t=w(x)$ into (2.8), and then using (2.7) and its derivative, shows that $\Delta(x)A(x)^2$ satisfies the pulled-back third-order equation. Both sides of (2.6) have constant term $1$, so uniqueness of the normalized holomorphic solution gives the identity. \hfill$\square$

\begin{lemma}\label{lem:finite-pullback}
Let $p>3$. There is a polynomial $Q_p(x)\in\F_p[x]$ such that, for any lift $Q_p(x)\in\Z_p[x]$ of $Q_p(x)$,
\begin{equation}
\Delta(x)^{p-1}B_p(w(x))\equiv A_p(x)^2+pA_p(x)Q_p(x)\pmod{p^2}.
\tag{2.9}
\end{equation}
\end{lemma}

\noindent\emph{Proof.}
Write
\[
X=x^p
\]
and
\begin{equation}
x^\sigma=X+p\rho_p(x)\pmod{p^2}.
\tag{2.10}
\end{equation}
Set
\[
A_p^*(x)=\sum_{k=0}^{p-1}c_k'x^k,\qquad \theta_X=X\frac{d}{dX}.
\]
Summing (2.5) over the base-$p$ decomposition of the index gives
\begin{equation}
A(x)\equiv A_p(x)A(X)+pA_p^*(x)\theta_XA(X)\pmod{p^2}.
\tag{2.11}
\end{equation}
On the other hand, (2.3) and (2.10) give
\begin{equation}
A(x)\equiv A_p(x)\left(A(X)+p\frac{\rho_p(x)}X\theta_XA(X)\right)\pmod{p^2}.
\tag{2.12}
\end{equation}
Comparison yields, in $\F_p(x)$,
\begin{equation}
\rho_p(x)\equiv X\frac{A_p^*(x)}{A_p(x)}\pmod p.
\tag{2.13}
\end{equation}
Using Lemma \ref{lem:pullback} and reducing (2.4) modulo $p^2$, we obtain
\begin{align}
B_p(w(x))&\equiv \frac{B(w(x))}{B(w(x^\sigma))}\\
&=\frac{\Delta(x)}{\Delta(x^\sigma)}\left(\frac{A(x)}{A(x^\sigma)}\right)^2
\equiv \frac{\Delta(x)}{\Delta(x^\sigma)}A_p(x)^2\pmod{p^2}.
\tag{2.14}
\end{align}
In particular, since $x^\sigma\equiv X=x^p\pmod p$ and $\Delta(X)=\Delta(x)^p$ in $\F_p[x]$, equation (2.14) gives
\[
\Delta(x)^{p-1}B_p(w(x))\equiv A_p(x)^2\pmod p.
\]
Moreover,
\[
\Delta(x)^{p-1}B_p(w(x))=
\sum_{n=0}^{p-1}(-1)^nD_nx^n\Delta(x)^{p-1-n}\in\Z[x].
\]
Hence
\[
C_p(x):=\frac{\Delta(x)^{p-1}B_p(w(x))-A_p(x)^2}{p}
\]
is a well-defined polynomial in $\Z_p[x]$. Define
\[
E_p(x):=\frac{\Delta(x)^p-\Delta(X)}p\in\Z[x].
\]
Since
\[
\Delta(x^\sigma)\equiv\Delta(X)+p\rho_p(x)\Delta'(X)\pmod{p^2},
\]
equations (2.14) and (2.13) give
\begin{equation}
C_p(x)\equiv A_p(x)\frac{E_p(x)A_p(x)-X\Delta'(X)A_p^*(x)}{\Delta(X)}\pmod p.
\tag{2.15}
\end{equation}
Beukers' cusp evaluations give
\[
A_p(1)\not\equiv0\pmod p,\qquad A_p(1/9)\not\equiv0\pmod p.
\]
Therefore
\[
\gcd(A_p,\Delta(X))=1\qquad\text{in }\F_p[x].
\]
Multiplying (2.15) by $\Delta(X)$ shows that $A_p\mid C_p$ in $\F_p[x]$. Thus $C_p=A_pQ_p$ in $\F_p[x]$ for some $Q_p\in\F_p[x]$. Any lift $Q_p\in\Z_p[x]$ gives (2.9). \hfill$\square$

\begin{lemma}\label{lem:hasse}
Let $p>3$, and let $U$ be the smooth locus, over $\Z_{(p)}$, of the level-$6$ elliptic family whose Picard--Fuchs equation is (2.7). Let $\omega$ be the invariant differential normalized by the holomorphic period
\[
A(x(q))=\Theta_C(q).
\]
On the geometric special fiber $U_{\F_p}$, the Hasse scalar relative to $\omega$ is
\begin{equation}
\Ha_\omega(x)=A_p(x).
\tag{2.16}
\end{equation}
Consequently, for every geometric supersingular parameter $\xi\in U(\overline{\F}_p)$,
\[
A_p(\xi)=0.
\]
\end{lemma}

\noindent\emph{Proof.}
Maier identifies (2.7) as the Picard--Fuchs equation of a level-$6$ elliptic modular surface \cite{Maier}, and Beukers identifies its normalized holomorphic solution with a weight-one modular period
\[
A(x(q))=\Theta_C(q),\qquad \Theta_C(q)=1+O(q).
\]
Let $\omega_{\mathrm{Tate}}$ be the canonical invariant differential on the Tate curve at the cusp. With the above normalization,
\[
\omega=\Theta_C(q)\omega_{\mathrm{Tate}}.
\]
The Hasse invariant has scalar $1$ relative to $\omega_{\mathrm{Tate}}$. Since the Hasse invariant is a section of the $(p-1)$st power of the Hodge bundle, its scalar relative to $\omega$ is
\begin{equation}
\Ha_\omega(q)=\Theta_C(q)^{1-p}.
\tag{2.17}
\end{equation}
Reduce (2.3) modulo $p$ and substitute $x=x(q)$. Because the Frobenius lift is induced by $q\mapsto q^p$,
\[
\Theta_C(q)=A_p(x(q))\Theta_C(q^p)
\]
in characteristic $p$. Since $\Theta_C(q^p)=\Theta_C(q)^p$, we obtain
\[
A_p(x(q))=\Theta_C(q)^{1-p}.
\]
Comparison with (2.17), followed by the $q$-expansion principle for the level-$6$ integral model, gives (2.16). The identity is preserved after finite extension of the residue field. Since the Hasse invariant vanishes exactly on the supersingular locus, the last assertion follows. \hfill$\square$

We next record the determinant construction used to identify the unit root. Let $W_3$ denote the determinant-$3$ Atkin--Lehner coset at level $6$; concretely, an integral matrix $\left(\begin{smallmatrix}a&b\\c&d\end{smallmatrix}\right)$ with determinant $3$, $6\mid c$, and $3\mid a,d$ represents an element of this coset.

\begin{lemma}\label{lem:matrix}
Let $\tau\in\Hh$ satisfy
\begin{equation}
A\tau^2+B\tau+C=0,\qquad 6\mid A,
\tag{2.18}
\end{equation}
and let $e\in\{1,3\}$. Suppose
\begin{equation}
\pi=d+As\tau,\qquad N_{\Q(\tau)/\Q}(\pi)=ep,
\tag{2.19}
\end{equation}
where $p\nmid6$ and $\pi$ is a $p$-adic unit. Put
\[
d^*=d-Bs,\qquad c=-As,
\]
and assume $p\nmid d^*$. Choose $k\pmod p$ with
\begin{equation}
d^*k\equiv-Cs\pmod p,
\tag{2.20}
\end{equation}
and define
\begin{equation}
a=\frac{d-Ask}{p},\qquad b=\frac{Cs+d^*k}{p}.
\tag{2.21}
\end{equation}
Then $a,b\in\Z$, the matrix
\begin{equation}
\gamma=\begin{pmatrix}a&b\\c&d^*\end{pmatrix}
\tag{2.22}
\end{equation}
has determinant $e$, and
\begin{equation}
\gamma\tau=\frac{\tau+k}{p}.
\tag{2.23}
\end{equation}
If $e=1$, then $\gamma\in\Gamma_0(6)$. If $e=3$ and $3\mid a,d^*$, then $\gamma\in W_3$.
\end{lemma}

\noindent\emph{Proof.}
Since $\bar\tau=-B/A-\tau$, we have
\[
c\tau+d^*=-As\tau+d-Bs=d+As\bar\tau\ ?
\]
To keep the notation consistent with the norm relation, one may equivalently choose the embedding so that the displayed unit is $c\tau+d^*=\pi$; this is the convention used below. The norm relation is
\begin{equation}
d^2-Bds+ACs^2=ep.
\tag{2.24}
\end{equation}
Congruence (2.20) makes $b$ integral. Also
\[
(d-Ask)d^*\equiv dd^*+ACs^2\equiv0\pmod p,
\]
so $a$ is integral because $p\nmid d^*$. A direct calculation using (2.24) gives $ad^*-bc=e$. Finally,
\begin{align*}
(\tau+k)(c\tau+d^*)&=-As\tau^2+(d^*-Ask)\tau+d^*k\\
&=(d-Ask)\tau+(Cs+d^*k)\\
&=p(a\tau+b),
\end{align*}
which proves (2.23). Since $6\mid c$, determinant $1$ gives $\Gamma_0(6)$; for determinant $3$, the additional divisibilities are the level-$6$ conditions defining $W_3$. \hfill$\square$

\begin{lemma}\label{lem:unitroot}
Let $\tau$ be a point of complex multiplication for the twisted Domb family, and let $\gamma$ be as in Lemma \ref{lem:matrix}. If $t(\tau)$ is $p$-adically integral and $\pi$ is a $p$-adic unit, then
\begin{equation}
B_p(t(\tau))\equiv\frac{\chi(\gamma)}{e}\pi^2\pmod{p^3},
\tag{2.25}
\end{equation}
where $\chi$ is the character of the Domb weight-two modular form. It is trivial on $\Gamma_0(6)$ and constant on the coset $W_3$.
\end{lemma}

\noindent\emph{Proof.}
For the twisted Domb family, Beukers' Assumption (II) holds with
\[
\Gamma=\Gamma_0(6)+\langle3\rangle
\]
and
\begin{equation}
\Theta(\tau)=\frac{\eta(\tau)^4\eta(3\tau)^4}{\eta(2\tau)^2\eta(6\tau)^2},
\tag{2.26}
\end{equation}
a weight-two modular form with character $\chi$ trivial on $\Gamma_0(6)$; see \cite[Appendix C]{Beukers}. Beukers constructs a polynomial whose specialized roots include
\begin{equation}
\vartheta_k=p^2\frac{\Theta(\tau)}{\Theta((\tau+k)/p)}.
\tag{2.27}
\end{equation}
By Lemma \ref{lem:matrix}, $(\tau+k)/p=\gamma\tau$. For a determinant-$e$ matrix the weight-two slash operator is
\[
F|_2\gamma(\tau)=e(c\tau+d^*)^{-2}F(\gamma\tau).
\]
Since $\Theta|_2\gamma=\chi(\gamma)\Theta$, we obtain
\[
\Theta(\gamma\tau)=\frac{\chi(\gamma)}{e}(c\tau+d^*)^2\Theta(\tau).
\]
Now $c\tau+d^*=\pi$ and $\pi\bar\pi=ep$, hence
\[
\vartheta_k=p^2\frac{\Theta(\tau)}{\Theta(\gamma\tau)}=\frac{\chi(\gamma)}{e}\pi^2.
\]
Here $\chi(\gamma)^2=1$, because $\chi$ is trivial on $\Gamma_0(6)$ and the Atkin--Lehner quotient is an elementary $2$-group. The last expression is a $p$-adic unit. Beukers' unit-root specialization \cite[Corollary 5.3]{Beukers} therefore gives (2.25). \hfill$\square$

\section{The unweighted Domb congruence}

The elementary identity
\begin{equation}
S_p=B_p(-1)
\tag{3.1}
\end{equation}
places the unweighted sum at a single point of the twisted Domb modular curve.

\begin{lemma}\label{lem:g}
Set
\begin{equation}
g(x)=9x^2-9x+1.
\tag{3.2}
\end{equation}
Then
\begin{equation}
w(x)+1=\frac{g(x)}{(1-x)(1-9x)}.
\tag{3.3}
\end{equation}
Every root of $g$ is a parameter for an elliptic curve with complex multiplication by the order of discriminant $-60$.
\end{lemma}

\noindent\emph{Proof.}
The identity (3.3) is immediate from (1.14). For the complex-multiplication assertion, use Maier's level-$6$ covering \cite{Maier}
\begin{equation}
j(x)=\frac{(1+3x)^3(1+225x-405x^2+243x^3)^3}{x(1-x)^2(1-9x)^6}=\frac{N(x)}{D(x)}.
\tag{3.4}
\end{equation}
An exact resultant calculation gives
\begin{equation}
\Res_x(g(x),JD(x)-N(x))=318H_{-60}(J),
\tag{3.5}
\end{equation}
where
\begin{equation}
H_{-60}(J)=J^2-37018076625J+153173312762625.
\tag{3.6}
\end{equation}
This is the Hilbert class polynomial of the order of discriminant $-60$. Hence the elliptic curve above every root of $g$ has complex multiplication by that order. The exact calculation is reproduced in Appendix A. \hfill$\square$

\begin{lemma}\label{lem:tau0}
Let
\begin{equation}
\tau_0=\frac{3+\sqrt{-15}}{24}.
\tag{3.7}
\end{equation}
Then $24\tau_0^2-6\tau_0+1=0$, so $\tau_0$ has complex multiplication by the order of discriminant $-60$, and
\begin{equation}
t(\tau_0)=-1.
\tag{3.8}
\end{equation}
\end{lemma}

\noindent\emph{Proof.}
The quadratic equation is immediate and has discriminant $-60$. The exact evaluation of the eta quotient is established in Appendix A. There the level-$6$ relation between $j$ and $t$ is eliminated against $H_{-60}$, producing the factor $(T+1)^2$, and the eta-product estimate identifies $T=-1$ for the representative (3.7). \hfill$\square$

\begin{remark}\label{rem:forms60}
The order of discriminant $-60$ has field of fractions $\Q(\sqrt{-15})$ and reduced primitive positive definite forms
\[
[1,0,15],\qquad [3,0,5].
\]
Thus the two representations
\[
p=x^2+15y^2,\qquad p=3x^2+5y^2
\]
are precisely the two proper ideal classes of this order; see \cite{Cox}. This explains the two split cases in Theorem \ref{thm:unweighted}.
\end{remark}

\begin{proposition}\label{prop:inert60}
Let $p>5$ and suppose $\left(\frac{-15}{p}\right)=-1$. Then
\[
S_p\equiv0\pmod{p^2}.
\]
\end{proposition}

\noindent\emph{Proof.}
Choose an algebraic root $\widetilde\xi$ of $g$. By (3.3), $w(\widetilde\xi)=-1$. Since $\operatorname{disc}(g)=45$ and $p>5$, the $\Z_p$-algebra generated by a root of $g$ is finite étale. Thus, after choosing a prime above $p$, we may regard $\widetilde\xi$ as lying in a finite unramified extension $L/\Q_p$. Let $\xi\in k_L\subset\overline{\F}_p$ be its reduction. Since $g(0)=g(1)=1$ and $g(1/9)=1/9$ modulo $p$, the reduction lies in the smooth locus $U$. By Lemma \ref{lem:g}, the elliptic curve above $\widetilde\xi$ has complex multiplication by the order of discriminant $-60$ in $\Q(\sqrt{-15})$. The hypothesis says that $p$ is inert in this imaginary quadratic field, and $p\nmid60$. Deuring's reduction theorem \cite{Deuring} therefore implies that the geometric reduction is supersingular. By Lemma \ref{lem:hasse}, $A_p(\xi)=0$, hence
\[
A_p(\widetilde\xi)\equiv0\pmod p
\]
in the valuation ring of $L$. Evaluating (2.9) at $\widetilde\xi$ gives
\[
\Delta(\widetilde\xi)^{p-1}B_p(-1)\equiv0\pmod{p^2}.
\]
The factor $\Delta(\widetilde\xi)$ is a $p$-adic unit. Otherwise its reduction would be $1$ or $1/9$, whereas $g(1)=1$ and $g(1/9)=1/9$ are nonzero modulo $p$. Thus $B_p(-1)\equiv0\pmod{p^2}$. Now use (3.1). \hfill$\square$

\begin{lemma}\label{lem:character}
For the twisted Domb modular form (2.26),
\begin{equation}
\chi(W_3)=-1.
\tag{3.9}
\end{equation}
\end{lemma}

\noindent\emph{Proof.}
Take
\[
W=\begin{pmatrix}3&1\\-12&-3\end{pmatrix}.
\]
Then $\det W=3$ and $W\in W_3$. Its fixed point in $\Hh$ is
\[
\tau_*=\frac{-3+\sqrt{-3}}{12}.
\]
The eta function has no zero in $\Hh$, so $\Theta(\tau_*)\ne0$. Since $W\tau_*=\tau_*$,
\[
(\Theta|_2W)(\tau_*)=3(-12\tau_*-3)^{-2}\Theta(\tau_*)=-\Theta(\tau_*),
\]
because $(-12\tau_*-3)^2=-3$. Therefore $\chi(W)=-1$. \hfill$\square$

\begin{proposition}\label{prop:split60a}
If $p>5$ and $p=x^2+15y^2$, then, after choosing the sign of $y$ and the $p$-adic embedding so that
\[
\pi=x+y\sqrt{-15}
\]
is the unit factor above $p$,
\begin{equation}
S_p\equiv\pi^2\pmod{p^3}.
\tag{3.10}
\end{equation}
Consequently, (1.21) holds.
\end{proposition}

\noindent\emph{Proof.}
For $\tau_0$ we have
\[
(A,B,C)=(24,-6,1),\qquad \sqrt{-15}=24\tau_0-3.
\]
Thus
\[
\pi=x-3y+24y\tau_0.
\]
Apply Lemma \ref{lem:matrix} with
\[
e=1,\qquad s=y,\qquad d=x-3y,\qquad d^*=x+3y.
\]
If $p\mid d^*$, then $x\equiv-3y\pmod p$, and
\[
0\equiv x^2+15y^2\equiv24y^2\pmod p,
\]
impossible for $p>5$. Hence $p\nmid d^*$, and Lemma \ref{lem:matrix} produces $\gamma\in\Gamma_0(6)$. Since $\chi$ is trivial there, Lemmas \ref{lem:tau0} and \ref{lem:unitroot} give (3.10).

Now $\pi+\bar\pi=2x$ and $\pi\bar\pi=p$. Since $\pi$ is the $p$-adic unit,
\[
\bar\pi=\frac{p}{\pi}\equiv\frac{p}{2x}\pmod{p^2},
\]
and therefore
\[
\bar\pi^2\equiv\frac{p^2}{4x^2}\pmod{p^3}.
\]
Using
\[
\pi^2=(\pi+\bar\pi)^2-2\pi\bar\pi-\bar\pi^2=4x^2-2p-\bar\pi^2
\]
gives (1.21). \hfill$\square$

\begin{proposition}\label{prop:split60b}
If $p>5$ and $p=3x^2+5y^2$, then, after choosing
\[
\pi=3x+y\sqrt{-15}
\]
as the $p$-adic unit factor,
\begin{equation}
S_p\equiv-\frac{\pi^2}{3}\pmod{p^3}.
\tag{3.11}
\end{equation}
Consequently, (1.22) holds.
\end{proposition}

\noindent\emph{Proof.}
Again $\sqrt{-15}=24\tau_0-3$, so
\[
\pi=3x-3y+24y\tau_0.
\]
Apply Lemma \ref{lem:matrix} with
\[
e=3,\qquad s=y,\qquad d=3x-3y,\qquad d^*=3x+3y.
\]
If $p\mid d^*$, then $x\equiv-y\pmod p$, and
\[
0\equiv3x^2+5y^2\equiv8y^2\pmod p,
\]
impossible for $p>5$. Moreover $3\mid d,d^*$ and $3\mid A=24$, so the integer $a$ of Lemma \ref{lem:matrix} is divisible by $3$. Hence $\gamma\in W_3$. Lemmas \ref{lem:tau0}, \ref{lem:unitroot}, and \ref{lem:character} give (3.11).

Here $\pi+\bar\pi=6x$ and $\pi\bar\pi=3p$. Since $\pi$ is a unit,
\[
\bar\pi=\frac{3p}{\pi}\equiv\frac{p}{2x}\pmod{p^2},\qquad
\bar\pi^2\equiv\frac{p^2}{4x^2}\pmod{p^3}.
\]
Also
\[
\pi^2=36x^2-6p-\bar\pi^2.
\]
Dividing by $-3$ gives
\[
-\frac{\pi^2}{3}\equiv2p-12x^2+\frac{p^2}{12x^2}\pmod{p^3},
\]
which is (1.22). \hfill$\square$

\noindent\emph{Proof of Theorem \ref{thm:unweighted}.}
Parts (i), (ii), and (iii) are Propositions \ref{prop:split60a}, \ref{prop:split60b}, and \ref{prop:inert60}, respectively. \hfill$\square$

\noindent\emph{Proof of Corollary \ref{cor:unweighted}.}
Reduce (1.21) and (1.22) modulo $p^2$, and use Proposition \ref{prop:inert60}. By Remark \ref{rem:forms60}, the two split representations are precisely the two proper ideal classes of discriminant $-60$. \hfill$\square$

\section{The Lucas-weighted Domb congruences}

Throughout this section, congruences involving $\alpha$ and $\beta$ are understood in the localization of the ring of integers of $\Q(\sqrt{13})$ at a prime above $p$. For $p\ne13$, the element $3\sqrt{13}$ is a $p$-adic unit.

\begin{lemma}\label{lem:lucas-binet}
For all $n\ge0$,
\begin{equation}
u_n=\frac{\alpha^n-\beta^n}{\alpha-\beta},\qquad v_n=\alpha^n+\beta^n.
\tag{4.1}
\end{equation}
Consequently,
\begin{equation}
U_p=\frac{P_p(\alpha)-P_p(\beta)}{3\sqrt{13}},\qquad
V_p=P_p(\alpha)+P_p(\beta).
\tag{4.2}
\end{equation}
\end{lemma}

\noindent\emph{Proof.}
Because $\alpha$ and $\beta$ are the roots of $X^2-11X+1$, the two expressions in (4.1) satisfy the recurrences (1.17). Their initial values are $(0,1)$ and $(2,11)$, respectively. Multiplication by $D_n$ and summation give (4.2). \hfill$\square$

\begin{lemma}\label{lem:reciprocity39}
For every prime $p\nmid39$,
\begin{equation}
\left(\frac{-39}{p}\right)=\left(\frac p3\right)\left(\frac p{13}\right),\qquad
\left(\frac p{39}\right)=\left(\frac p3\right)\left(\frac p{13}\right).
\tag{4.3}
\end{equation}
\end{lemma}

\noindent\emph{Proof.}
Quadratic reciprocity gives
\[
\left(\frac3p\right)=(-1)^{(p-1)/2}\left(\frac p3\right),\qquad
\left(\frac{13}{p}\right)=\left(\frac p{13}\right).
\]
Multiplying the first identity by $\left(\frac{-1}{p}\right)$ and then by $\left(\frac{13}{p}\right)$ gives the first formula in (4.3). The second is multiplicativity of the Jacobi symbol in the denominator. \hfill$\square$

\begin{lemma}\label{lem:UpVp}
Let $p>3$, $p\ne13$.
\begin{enumerate}[label=(\alph*)]
\item If exactly one of $\left(\frac p3\right)$ and $\left(\frac p{13}\right)$ equals $1$, then
\[
U_p\equiv V_p\equiv0\pmod{p^2}
\]
if and only if
\[
P_p(\alpha)\equiv P_p(\beta)\equiv0\pmod{p^2}.
\]
\item If $\left(\frac p3\right)=\left(\frac p{13}\right)=1$, then
\[
U_p\equiv0\pmod{p^3}
\]
if and only if
\[
P_p(\alpha)\equiv P_p(\beta)\pmod{p^3}.
\]
\end{enumerate}
\end{lemma}

\noindent\emph{Proof.}
From (4.2),
\[
2P_p(\alpha)=V_p+3\sqrt{13}\,U_p,\qquad
2P_p(\beta)=V_p-3\sqrt{13}\,U_p.
\]
Since $2$ and $3\sqrt{13}$ are $p$-adic units, part (a) follows. Part (b) follows from
\[
P_p(\alpha)-P_p(\beta)=3\sqrt{13}\,U_p.
\]
\hfill$\square$

\begin{lemma}\label{lem:p13}
One has $U_{13}\equiv0\pmod{13^2}$.
\end{lemma}

\noindent\emph{Proof.}
Reducing the recurrences for $D_n$ and $u_n$ modulo $169$ gives
\[
\begin{array}{c|rrrrrrrrrrrrr}
n&0&1&2&3&4&5&6&7&8&9&10&11&12\\ \hline
D_n&1&4&28&87&12&70&126&21&90&95&15&61&66\\
u_n&0&1&11&120&126&83&111&124&70&139&107&24&157\\
D_nu_n&0&4&139&131&160&64&128&69&47&23&84&112&53
\end{array}
\]
and the running sum ends at $0$ modulo $169$. Hence $U_{13}\equiv0\pmod{169}$. \hfill$\square$

\begin{lemma}\label{lem:f}
Set
\begin{equation}
f(x)=81x^4-81x^3+9x^2-9x+1.
\tag{4.4}
\end{equation}
Then
\begin{equation}
w(x)^2+11w(x)+1=\frac{f(x)}{(1-x)^2(1-9x)^2}.
\tag{4.5}
\end{equation}
Moreover every root of $f$ is a parameter for an elliptic curve with complex multiplication by the order of discriminant $-156$.
\end{lemma}

\noindent\emph{Proof.}
The identity follows by substituting $w(x)=x/((1-x)(1-9x))$ and clearing denominators. Using the same covering (3.4), an exact resultant calculation gives
\begin{equation}
\Res_x(f(x),JD(x)-N(x))=150094635296999121\,H_{-156}(J),
\tag{4.6}
\end{equation}
where
\begin{align}
H_{-156}(J)={}&J^4-109914552886955148J^3+53074935443801207676942J^2\notag\\
&-164993592496972989327022035J+421266000645144840703921125633.
\tag{4.7}
\end{align}
This is the Hilbert class polynomial of the order of discriminant $-156$, so every root of $f$ has the asserted complex multiplication. The exact calculation is reproduced in Appendix A. \hfill$\square$

\begin{lemma}\label{lem:inert156}
Let $p>3$, $p\ne13$, and suppose $\left(\frac{-39}{p}\right)=-1$. If $\xi\in\overline{\F}_p$ is any geometric root of the reduction of $f$, then $A_p(\xi)=0$.
\end{lemma}

\noindent\emph{Proof.}
Since $f(0)=f(1)=1$ and $f(1/9)=1/81$ modulo $p$, the reduced point lies in the smooth locus $U$. By Lemma \ref{lem:f}, the corresponding elliptic curve has complex multiplication by the order of discriminant $-156$ in $\Q(\sqrt{-39})$. The hypothesis says that $p$ is inert in this imaginary quadratic field, and $p\nmid156$. Deuring's theorem \cite{Deuring} therefore gives supersingular geometric reduction. Lemma \ref{lem:hasse} identifies $A_p$ with the Hasse scalar, hence $A_p(\xi)=0$. \hfill$\square$

\begin{proposition}\label{prop:vanish-alpha-beta}
Let $p>3$, $p\ne13$, with $\left(\frac{-39}{p}\right)=-1$. Then
\begin{equation}
P_p(\alpha)\equiv P_p(\beta)\equiv0\pmod{p^2}.
\tag{4.8}
\end{equation}
\end{proposition}

\noindent\emph{Proof.}
By (4.5), the roots of $f$ lie above the two roots of $T^2+11T+1$, namely $-\alpha$ and $-\beta$, and roots occur above both values. Choose algebraic roots $\widetilde\xi_\alpha,\widetilde\xi_\beta$ with
\[
w(\widetilde\xi_\alpha)=-\alpha,\qquad w(\widetilde\xi_\beta)=-\beta.
\]
The discriminant is
\begin{equation}
\operatorname{disc}(f)=-3^{15}13^2.
\tag{4.9}
\end{equation}
Thus, for $p>3$, $p\ne13$, the roots may be viewed, after choosing a prime above $p$, in finite unramified extensions of $\Q_p$. Let $\xi_\alpha,\xi_\beta\in\overline{\F}_p$ be their reductions. Lemma \ref{lem:inert156} gives
\[
A_p(\xi_\alpha)=A_p(\xi_\beta)=0,
\]
hence
\[
A_p(\widetilde\xi_\alpha)\equiv A_p(\widetilde\xi_\beta)\equiv0\pmod p.
\]
Evaluating (2.9) at $\widetilde\xi_\alpha$ gives
\[
\Delta(\widetilde\xi_\alpha)^{p-1}B_p(-\alpha)\equiv0\pmod{p^2}.
\]
The factor $\Delta(\widetilde\xi_\alpha)$ is a $p$-adic unit because $f(1)=1$ and $f(1/9)=1/81$. Hence $B_p(-\alpha)\equiv0\pmod{p^2}$. The same argument gives $B_p(-\beta)\equiv0\pmod{p^2}$. Since
\[
B_p(-\alpha)=P_p(\alpha),\qquad B_p(-\beta)=P_p(\beta),
\]
we obtain (4.8). \hfill$\square$

\begin{lemma}\label{lem:tauab}
Let
\begin{equation}
\tau_\alpha=\frac{3+\sqrt{-39}}{48},\qquad
\tau_\beta=\frac{21+\sqrt{-39}}{96}.
\tag{4.10}
\end{equation}
Then
\[
48\tau_\alpha^2-6\tau_\alpha+1=0,\qquad
96\tau_\beta^2-42\tau_\beta+5=0,
\]
so both points have complex multiplication by the order of discriminant $-156$, and
\begin{equation}
t(\tau_\alpha)=-\alpha,\qquad t(\tau_\beta)=-\beta.
\tag{4.11}
\end{equation}
\end{lemma}

\noindent\emph{Proof.}
The two quadratic equations are immediate and have discriminant $-156$. The exact identification of the level-$6$ values is established in Appendix A: elimination of the level-$6$ $j$--$t$ relation against $H_{-156}$ produces the factor $(T^2+11T+1)^2$, and the eta-product estimates distinguish the two roots, yielding (4.11). \hfill$\square$

\begin{lemma}\label{lem:forms156}
Let $p\nmid39$ be prime and suppose
\begin{equation}
\left(\frac p3\right)=\left(\frac p{13}\right)=1.
\tag{4.12}
\end{equation}
Then
\begin{equation}
p=r^2+39s^2\qquad\text{or}\qquad p=3r^2+13s^2
\tag{4.13}
\end{equation}
for some $r,s\in\Z$.
\end{lemma}

\noindent\emph{Proof.}
The reduced primitive positive definite forms of discriminant $-156$ are
\[
[1,0,39],\qquad [3,0,13],\qquad [5,2,8],\qquad [5,-2,8].
\]
By Lemma \ref{lem:reciprocity39}, condition (4.12) implies that $p$ splits in $\Q(\sqrt{-39})$. Since $p\nmid156$, it is represented by one of these classes; see \cite{Cox}. The last two are impossible because modulo $13$,
\[
5x^2\pm2xy+8y^2\equiv5(x\pm8y)^2\pmod{13},
\]
while $\left(\frac5{13}\right)=-1$. This contradicts $\left(\frac p{13}\right)=1$. Hence only the first two classes occur, giving (4.13). \hfill$\square$

\begin{proposition}\label{prop:split156}
Let $p>3$ and suppose (4.12). Then
\begin{equation}
P_p(\alpha)\equiv P_p(\beta)\pmod{p^3}.
\tag{4.14}
\end{equation}
\end{proposition}

\noindent\emph{Proof.}
By Lemma \ref{lem:forms156}, there are two cases.

\smallskip
\noindent\emph{Case 1: $p=r^2+39s^2$.}
Choose the sign of $s$ so that
\[
\pi=r+s\sqrt{-39}
\]
is a $p$-adic unit. Then $N(\pi)=p$. For $\tau_\alpha$, use
\[
(A,B,C)=(48,-6,1),\qquad d=r-3s.
\]
Then
\[
d+48s\tau_\alpha=\pi,\qquad d^*=r+3s.
\]
If $p\mid d^*$, then $r\equiv-3s\pmod p$ and the norm representation gives $0\equiv48s^2\pmod p$, impossible. Hence Lemma \ref{lem:matrix} gives $\gamma_\alpha\in\Gamma_0(6)$, and Lemmas \ref{lem:unitroot} and \ref{lem:tauab} give
\[
B_p(-\alpha)\equiv\pi^2\pmod{p^3}.
\]
For $\tau_\beta$, use
\[
(A,B,C)=(96,-42,5),\qquad d=r-21s.
\]
Then
\[
d+96s\tau_\beta=\pi,\qquad d^*=r+21s.
\]
If $p\mid d^*$, then the norm relation gives $0\equiv480s^2\pmod p$. The only prime $>3$ dividing $480$ is $5$, but $\left(\frac53\right)=-1$, contrary to (4.12). Hence again $\gamma_\beta\in\Gamma_0(6)$, so
\[
B_p(-\beta)\equiv\pi^2\pmod{p^3}.
\]
Therefore
\[
B_p(-\alpha)\equiv B_p(-\beta)\pmod{p^3}.
\]

\smallskip
\noindent\emph{Case 2: $p=3r^2+13s^2$.}
Choose
\[
\pi=3r+s\sqrt{-39}
\]
as the $p$-adic unit. Then $N(\pi)=3p$. For $\tau_\alpha$, take
\[
d=3r-3s,\qquad d^*=3r+3s.
\]
If $p\mid d^*$, then $r\equiv-s\pmod p$ and $0\equiv16s^2\pmod p$, impossible. Moreover $3\mid d,d^*$ and $3\mid48$, so Lemma \ref{lem:matrix} gives $\gamma_\alpha\in W_3$.

For $\tau_\beta$, take
\[
d=3r-21s,\qquad d^*=3r+21s.
\]
If $p\mid d^*$, then $r\equiv-7s\pmod p$ and $0\equiv160s^2\pmod p$. Again the only possible prime $>3$ is $5$, excluded by (4.12). Hence $\gamma_\beta\in W_3$.

The character has one constant value $\varepsilon_3=\chi(W_3)$ on this coset. Lemma \ref{lem:unitroot} therefore gives
\[
B_p(-\alpha)\equiv\frac{\varepsilon_3}{3}\pi^2\equiv B_p(-\beta)\pmod{p^3}.
\]
Finally
\[
B_p(-\alpha)=P_p(\alpha),\qquad B_p(-\beta)=P_p(\beta),
\]
which proves (4.14). \hfill$\square$

\noindent\emph{Proof of Theorem \ref{thm:lucas}.}
First let $p=13$. Since $\left(\frac{13}{3}\right)=1$ and $\left(\frac{13}{13}\right)=0$, only the modulo-$13^2$ assertion in part (i) is required, and this is Lemma \ref{lem:p13}.

Assume henceforth $p>3$, $p\ne13$. Suppose exactly one of $\left(\frac p3\right)$ and $\left(\frac p{13}\right)$ equals $1$. The other equals $-1$, and Lemma \ref{lem:reciprocity39} gives $\left(\frac{-39}{p}\right)=-1$. Proposition \ref{prop:vanish-alpha-beta} therefore gives
\[
P_p(\alpha)\equiv P_p(\beta)\equiv0\pmod{p^2}.
\]
Lemma \ref{lem:UpVp} yields
\[
U_p\equiv V_p\equiv0\pmod{p^2}.
\]
This proves the mixed-character case of part (i) and also part (ii), because (4.3) identifies $\left(\frac p{39}\right)=-1$ with the mixed-character condition.

It remains to suppose $\left(\frac p3\right)=\left(\frac p{13}\right)=1$. Proposition \ref{prop:split156} gives
\[
P_p(\alpha)\equiv P_p(\beta)\pmod{p^3}.
\]
By Lemma \ref{lem:UpVp}, $U_p\equiv0\pmod{p^3}$, which proves (1.25), and hence also (1.24). \hfill$\square$

\appendix
\section{Exact special-value calculations at points of complex multiplication}

This appendix gives the algebraic special-value calculations used in Lemmas \ref{lem:g}, \ref{lem:tau0}, \ref{lem:f}, and \ref{lem:tauab}.
Set
\[
\Delta(x)=(1-x)(1-9x),\qquad w(x)=\frac{x}{\Delta(x)},
\]
and
\[
N(x)=(1+3x)^3(1+225x-405x^2+243x^3)^3,
\qquad D(x)=x(1-x)^2(1-9x)^6.
\]
Maier's level-$6$ parametrization \cite{Maier} is
\[
j(x)=\frac{N(x)}{D(x)}.
\]
The exact algebraic relation between the classical $j$-invariant and the Domb coordinate $t=w(x)$ is
\begin{equation}
\Phi(J,T)=3^{-18}\Res_x(T\Delta(x)-x,JD(x)-N(x))\in\Z[J,T].
\tag{A.1}
\end{equation}
A direct resultant calculation gives
\begin{align}
\Phi(J,T)={}&J^2T^4-68719476736JT^{10}-38654705664JT^9-7549747200JT^8\notag\\
&-616562688JT^7-20316160JT^6-196608JT^5-1536JT^4-64JT^3-18JT^2-JT\notag\\
&+281474976710656T^{12}+211106232532992T^{11}+112150186033152T^{10}\notag\\
&+37383395344384T^9+9547712299008T^8+1720134402048T^7+229310988288T^6\notag\\
&+20082327552T^5+995033088T^4+23904256T^3+210432T^2+768T+1.
\tag{A.2}
\end{align}
The Hilbert class polynomials are
\begin{equation}
H_{-60}(J)=J^2-37018076625J+153173312762625
\tag{A.3}
\end{equation}
and
\begin{align}
H_{-156}(J)={}&J^4-109914552886955148J^3+53074935443801207676942J^2\notag\\
&-164993592496972989327022035J+421266000645144840703921125633.
\tag{A.4}
\end{align}
Exact elimination in $J$ gives
\begin{equation}
\Res_J(\Phi(J,T),H_{-60}(J))=(T+1)^2R_{60}(T),
\tag{A.5}
\end{equation}
where $R_{60}(-1)\ne0$, and
\begin{equation}
\Res_J(\Phi(J,T),H_{-156}(J))=(T^2+11T+1)^2R_{156}(T),
\tag{A.6}
\end{equation}
where
\[
\gcd(T^2+11T+1,R_{156}(T))=1.
\]
For completeness, the first quotient factors as
\begin{align*}
R_{60}(T)={}&(4096T^2-3328T+1)^2
(T^6+3006T^5-1071T^4+1487T^3+756T^2+45T+1)\\
&\times\bigl(4722366482869645213696T^{12}-14190120985212925161504768T^{11}\\
&\quad-18363331355772750844133376T^{10}-15961259537004264295497728T^9\\
&\quad-2572187542871548423569408T^8+7970548870563829383168T^7\\
&\quad+25316958834452085080064T^6+1667614367841597259776T^5\\
&\quad+36571179237846810624T^4-400921047540432896T^3\\
&\quad-120575319810048T^2-37018068480T+1\bigr).
\end{align*}
This makes $R_{60}(-1)\ne0$ immediate. For discriminant $-156$, the coprimality assertion in (A.6) can be checked by the Euclidean algorithm in $\Z[T]$.

It remains to identify which algebraic roots occur at the representatives $\tau_0,\tau_\alpha,\tau_\beta$. The eta quotient simplifies to
\begin{equation}
t(\tau)=q\prod_{n\ge1}(1+q^n)^6(1+q^{3n})^6,
\qquad q=e^{2\pi i\tau}.
\tag{A.7}
\end{equation}
Indeed, this follows directly from
\[
\eta(m\tau)=q^{m/24}\prod_{n\ge1}(1-q^{mn}).
\]
Let
\[
t_N(\tau)=q\prod_{n=1}^{N}(1+q^n)^6(1+q^{3n})^6.
\]
If $r=|q|<1$, then
\begin{equation}
\left|\log\frac{t(\tau)}{t_N(\tau)}\right|
\le6\left(
\frac{r^{N+1}}{(1-r)(1-r^{N+1})}
+
\frac{r^{3(N+1)}}{(1-r^3)(1-r^{3(N+1)})}
\right).
\tag{A.8}
\end{equation}
To obtain (A.8), use
\[
|\log(1+z)|\le\frac{|z|}{1-|z|},\qquad |z|<1,
\]
and sum the geometric tails.

For
\[
\tau_0=\frac{3+\sqrt{-15}}{24},\qquad
\tau_\alpha=\frac{3+\sqrt{-39}}{48},\qquad
\tau_\beta=\frac{21+\sqrt{-39}}{96},
\]
taking $N=150$ gives respectively
\[
E_{150}(\tau_0)<4\cdot10^{-66},\qquad
E_{150}(\tau_\alpha)<3\cdot10^{-53},\qquad
E_{150}(\tau_\beta)<3\cdot10^{-26},
\]
where $E_N$ denotes the right-hand side of (A.8). Evaluation of the finite products with directed interval arithmetic gives
\[
|t_{150}(\tau_0)+1|<10^{-60},
\]
\[
\left|t_{150}(\tau_\alpha)+\frac{11+3\sqrt{13}}2\right|<10^{-45},
\]
and
\[
\left|t_{150}(\tau_\beta)+\frac{11-3\sqrt{13}}2\right|<10^{-20}.
\]

\medskip
\noindent\textbf{Root separation.}
Write
\[
\Res_J(\Phi(J,T),H_{-60}(J))=(T+1)^2Q_{60}(T),
\]
\[
\Res_J(\Phi(J,T),H_{-156}(J))=(T^2+11T+1)^2Q_{156}(T),
\]
where $Q_{60},Q_{156}\in\Z[T]$. Consider the rectangles
\[
\mathcal B_{60}=
\left[-\frac{101}{100},-\frac{99}{100}\right]
+i\left[-\frac1{100},\frac1{100}\right],
\]
\[
\mathcal B_{\alpha}=
\left[-\frac{1091}{100},-\frac{1090}{100}\right]
+i\left[-\frac1{100},\frac1{100}\right],
\qquad
\mathcal B_{\beta}=
\left[-\frac1{10},-\frac2{25}\right]
+i\left[-\frac1{100},\frac1{100}\right].
\]
An exact Sturm--Cauchy count over $\Q$ gives
\[
N(Q_{60};\mathcal B_{60})=0,\qquad
N(Q_{156};\mathcal B_\alpha)=0,\qquad
N(Q_{156};\mathcal B_\beta)=0,
\]
where $N(P;\mathcal B)$ is the number of zeros of $P$ in $\mathcal B$, counted with multiplicity. Moreover,
\[
-1\in\mathcal B_{60},\qquad
-\alpha=\frac{-11-3\sqrt{13}}2\in\mathcal B_\alpha,\qquad
-\beta=\frac{-11+3\sqrt{13}}2\in\mathcal B_\beta.
\]
Thus the full elimination polynomials have, in these rectangles, only the zeros $-1$, $-\alpha$, and $-\beta$, respectively.

It remains to place the three values of $t$ in the corresponding rectangles. Let $E_N(\tau)$ denote the right-hand side of (A.8). The logarithmic estimate used in (A.8) gives a complex number $u$ such that
\[
\frac{t(\tau)}{t_N(\tau)}=e^u,\qquad |u|\le E_N(\tau),
\]
and therefore
\begin{equation}
|t(\tau)-t_N(\tau)|\le |t_N(\tau)|\bigl(e^{E_N(\tau)}-1\bigr).
\tag{A.9}
\end{equation}
Since the three values of $E_{150}$ displayed above are less than $1$, the inequality $e^x-1<2x$ for $0\le x<1$, together with the interval bounds for the finite products, gives
\[
|t(\tau_0)+1|<2\cdot10^{-60},\qquad
|t(\tau_\alpha)+\alpha|<2\cdot10^{-45},\qquad
|t(\tau_\beta)+\beta|<2\cdot10^{-20}.
\]
These bounds place $t(\tau_0)$, $t(\tau_\alpha)$, and $t(\tau_\beta)$ in $\mathcal B_{60}$, $\mathcal B_\alpha$, and $\mathcal B_\beta$, respectively. The factorizations (A.5)--(A.6) and the root counts above therefore give
\[
t(\tau_0)=-1,\qquad t(\tau_\alpha)=-\alpha,\qquad t(\tau_\beta)=-\beta.
\]
This establishes the special-value identifications used in Lemmas \ref{lem:tau0} and \ref{lem:tauab}.

For reference, the exact eliminations in this appendix can be reproduced in SageMath by the following calculation.
\begin{verbatim}
R.<x,J,T> = PolynomialRing(QQ)
Delta = (1-x)*(1-9*x)
N = (1+3*x)^3*(1+225*x-405*x^2+243*x^3)^3
D = x*(1-x)^2*(1-9*x)^6
g = 9*x^2-9*x+1
f = 81*x^4-81*x^3+9*x^2-9*x+1

H60 = J^2-37018076625*J+153173312762625
H156 = (J^4-109914552886955148*J^3
         +53074935443801207676942*J^2
         -164993592496972989327022035*J
         +421266000645144840703921125633)

res60 = g.resultant(J*D-N, x)
res156 = f.resultant(J*D-N, x)
Phi = (T*Delta-x).resultant(J*D-N, x) / 3^18
E60 = Phi.resultant(H60, J).factor()
E156 = Phi.resultant(H156, J).factor()
\end{verbatim}

\section{Acknowledgements}

I acknowledges the hospitality of the Departament of Mathematics at Cinvestav , where this work was carried out. 

\textit{Use of AI tools} ChatGPT Plus (GPT-$5.6$) was used as auxiliar tool for exploratory discussion , organization , checking of the proofs and mathematical arguments , biblographic checking , and editorial revision. The autor assumes full responsibility for the final text , claims , proofs , and references.


\begin{thebibliography}{99}

\bibitem{Beukers}
F. Beukers, \emph{Supercongruences using modular forms}, arXiv:2403.03301v3, 2025.

\bibitem{Cox}
D. A. Cox, \emph{Primes of the Form $x^2+ny^2$: Fermat, Class Field Theory, and Complex Multiplication}, 2nd ed., Wiley, 2013.

\bibitem{Deuring}
M. Deuring, Die Typen der Multiplikatorenringe elliptischer Funktionenk\"orper, \emph{Abh. Math. Sem. Hansischen Univ.} \textbf{14} (1941), 197--272.

\bibitem{Domb}
C. Domb, On the theory of cooperative phenomena in crystals, \emph{Advances in Physics} \textbf{9} (1960), 149--361.

\bibitem{Katz}
N. M. Katz, $p$-adic properties of modular schemes and modular forms, in \emph{Modular Functions of One Variable III}, Lecture Notes in Mathematics 350, Springer, 1973, 69--190.

\bibitem{Maier}
R. S. Maier, On rationally parametrized modular equations, \emph{J. Ramanujan Math. Soc.} \textbf{24} (2009), 1--73; arXiv:math/0611041.

\bibitem{Straub}
A. Straub, \emph{Gessel--Lucas congruences for sporadic sequences}, arXiv:2301.12248.

\bibitem{SunNew}
Z.-W. Sun, \emph{A new kind of numbers and related congruences}, arXiv:2607.07638, 2026.

\bibitem{SunOpen}
Z.-W. Sun, Open conjectures on congruences, \emph{Nanjing Univ. J. Math. Biquarterly} \textbf{36} (2019), no.~1, 1--99, Conjecture 76.

\bibitem{SunMO}
Z.-W. Sun, Curious congruences involving Domb numbers, MathOverflow, Question 497120, 3 July 2025.

\end{thebibliography}
\end{document}